\documentclass[12pt,a4paper,reqno]{amsart} 
\usepackage[left=2cm,right=2cm,top=2.4cm,bottom=2.4cm,a4paper]{geometry}
\usepackage{amsmath,amscd,amssymb,latexsym}
\usepackage{graphicx}
\usepackage{enumerate}
\usepackage{stackengine}
\stackMath   
\usepackage{placeins}

\usepackage[all,cmtip]{xy}

\makeatletter \@namedef{subjclassname@2010}{
  \textup{2010} Mathematics Subject Classification}
\makeatother

\usepackage{color}
\definecolor{ultramarine}{rgb}{0.07, 0.04, 0.56}
\usepackage[pagebackref=true,colorlinks=true,linkcolor=ultramarine,
citecolor=ultramarine,urlcolor=forestgreen]{hyperref}

\begin{document}
\setlength{\unitlength}{1mm}

\newtheorem{thm}{Theorem}[section]
\newtheorem{lem}[thm]{Lemma}
\newtheorem{prop}[thm]{Proposition}
\newtheorem{defn}[thm]{Definition}
\newtheorem{cor}[thm]{Corollary}
\newtheorem*{rmk}{Remark}

\newcommand{\QQ}{\mathbb{Q}}
\newcommand{\ZZ}{\mathbb{Z}}
\newcommand{\RR}{\mathbb{R}}
\newcommand{\NN}{\mathbb{N}}
\newcommand{\CC}{\mathbb{C}}
\newcommand{\XX}{\mathbb{X}}

\newcommand{\cO}{\mathcal{O}}
\newcommand{\cC}{\mathcal{C}}
\newcommand{\cN}{\mathcal{N}}

\newcommand{\fa}{\mathfrak{a}}
\newcommand{\fp}{\mathfrak{p}}
\newcommand{\fF}{\mathfrak{F}}
\newcommand{\fM}{\mathfrak{M}}
\newcommand{\fX}{\mathfrak{X}}

\newcommand{\SL}{\mathrm{SL}}
\newcommand{\GL}{\mathrm{GL}}
\newcommand{\sign}{\mathrm{sign}}
\newcommand{\op}{\mathrm{op}}

\title[Congruences for odd quadratic class numbers]
{Congruences for odd class numbers of 
quadratic fields with odd discriminant}

\author{Jigu Kim and Yoshinori Mizuno}


\maketitle

\noindent{\small {\bf Abstract.}
For any distinct two primes $p_1\equiv p_2\equiv 3$ $(\text{mod }4)$, let $h(-p_1)$, $h(-p_2)$ and $h(p_1p_2)$ be the class numbers of the quadratic fields
$\mathbb{Q}(\sqrt{-p_1})$, $\mathbb{Q}(\sqrt{-p_2})$ and $\mathbb{Q}(\sqrt{p_1p_2})$, respectively.
Let $\omega_{p_1p_2}:=(1+\sqrt{p_1p_2})/2$ and let $\Psi(\omega_{p_1p_2})$ be the Hirzebruch sum of $\omega_{p_1p_2}$. 
We show that $h(-p_1)h(-p_2)\equiv h(p_1p_2)\Psi(\omega_{p_1p_2})/n$ $(\text{mod }8)$, where $n=6$ (respectively, $n=2$) if $\min\{p_1,p_2\}>3$ (respectively, otherwise). 
We also consider the real quadratic order with conductor $2$ in $\mathbb{Q}(\sqrt{p_1p_2})$. 
}

\section{Introduction and results}\label{sec-1}
In this paper $\Delta$ denotes a {\it quadratic discriminant}, which means that it is a non-square integer such that $\Delta\equiv 0$ or $1$ $(\text{mod }4)$.
It is called {\it fundamental} if it is the discriminant of a quadratic number field.
Every quadratic discriminant $\Delta$ is a unique square multiple of a fundamental discriminant and we write $\Delta=f^2d$, where $f$ is a positive integer and $d$ is the fundamental discriminant of the field $K=\mathbb{Q}(\sqrt{d})$.
By $h(\Delta)$ (respectively, $h^+(\Delta)$) we denote the {\it wide} (respectively, {\it narrow})
{\it class number}  of $\mathcal{O}_{\Delta}$ the quadratic order with {\it conductor} $f$ in $K$.
Let $\omega_\Delta:=(\sigma_\Delta+\sqrt{\Delta})/2$, where $\sigma_\Delta=0$ if $\Delta$ is even and $\sigma_\Delta=1$ if $\Delta$ is odd.
For $\Delta>0$, the {\it fundamental unit} of the quadratic order $\mathcal{O}_{\Delta}$ is denoted by $\varepsilon_{\Delta}$. We have
\begin{equation*}
\varepsilon_\Delta=\frac{t+u\sqrt{\Delta}}{2},
\end{equation*}
where $t$ and $u$ are positive integers satisfying $t\equiv u$
$(\text{mod }2)$.

Any real quadratic irrational $\xi$ has a continued fraction expansion of the form
{\scriptsize 
$$\xi=
\hat{v}_1+\frac{1}{\displaystyle
\hat{v}_2+\frac{1}{\displaystyle \,\,\,\,
\stackunder{}{\ddots\stackunder{}{
\,\,\displaystyle+\frac{1}{\displaystyle \hat{v}_k+\frac{1}{\displaystyle v_0
+\frac{1}{\displaystyle \,\,\,\,
\stackunder{}{\ddots\stackunder{}{
\,\,\displaystyle+\frac{1}{\displaystyle v_{l-1}+\frac{1}{\displaystyle v_0+
\,\,\,\,
\stackunder{}{\ddots\stackunder{}{}}
}}
}}}}}
}}}}.
$$}
\vspace{5pt}

\noindent
with the {\it basic period} $\{v_0,\cdots,v_{l-1}\}$.
We write
$\xi=[\hat{v}_1,\cdots,\hat{v}_k;\overline{v_0,\cdots,v_{l-1}}],$
and let
$$\Psi(\xi):=\left\{
\begin{array}{ll}\displaystyle\sum_{i=0}^{l-1}(-1)^{k+i}v_i &\text{if }l\text{ is even},\\
\displaystyle 0 &\text{if }l\text{ is odd}.
\end{array}\right.
$$
This number $\Psi(\xi)$ is said to be the {\it Hirzebruch sum} of $\xi$.
When the {\it pre-period} $\{\hat{v}_1,\cdots \hat{v}_k\}$ is empty, we understand $k=0$. 
In the 1970s, Hirzebruch and Zagier \cite{Hir73,Zag75} obtained a beautiful formula on the class number of an imaginary quadratic field in terms of the Hirzebruch sum\footnote{In \cite{CGPY15} and \cite{Miz21}, the authors used the Zagier sum, which is defined by a {\it negative continued fraction}. We note that the Hirzebruch sum is three times the Zagier sum, $3\mid \Psi(\omega_{4p})$ in Theorem \ref{CGPY-thm} and $3\mid \Psi(\omega_{16p})$ in Theorem \ref{Miz-thm} (for detail, see \cite[p. 312]{CG19}, \cite[p.353]{Miz21} and \cite[Satz 2]{Lan76}).} (for the formula, see Remark \ref{Remark 2.17}).
As an application of Zagier's formula, Chua et al. 
solved Richard Guy's conjecture (cf. \cite[p. 1347]{CGPY15}), which is equivalent to the following theorem.
\begin{thm}[{\cite[Theorem 1.3]{CGPY15}}]\label{CGPY-thm}
For any prime $p\equiv 3$ $(\text{mod }4)$ with $p>3$,
$$h(-p)\equiv h(4p)\frac{\Psi(\omega_{4p})}{3}\quad (\text{mod } 16).$$
\end{thm}

The result obtained by Zagier \cite{Zag75}, however, does not work in the case of primes congruent
to 1 modulo 4.
Cheng and Guo \cite{CG19} used results obtained by Lu \cite{Lu91}, and they proved similar congruences modulo $8$ (respectively, modulo $4$) if $p\equiv 1$ $(\text{mod 8})$ (respectively, if $p\equiv 5$ $(\text{mod 8})$ and $\varepsilon_{p}$ has integral coefficients).

The second named author and Kaneko \cite{KM20}
presented an explicit form of genus character $L$-functions of quadratic orders, and
they generalized Zagier's results (for detail, see Section \ref{sec-2}).
As a consequence, the second named author showed a reasonable counterpart of Theorem \ref{CGPY-thm}.

\begin{thm}[{\cite[Theorem 1]{Miz21}}]\label{Miz-thm}
For any prime $p\equiv 1$ $(\text{mod }4)$,
$$h(-4p)\equiv h(16p) \frac{\Psi(\omega_{16p})}{3}\quad (\text{mod } 16).$$
\end{thm}

\vspace{1pt}

We remark that for a prime $p\equiv 1$ $(\text{mod }4)$, $h(16p)$ the class number of the quadratic order $\mathcal{O}_{16p}$ equals to $3h(p)$ (respectively, $h(p)$) if $p\equiv 5$ $(\text{mod }8)$ and $\varepsilon_{p}$ has integral coefficients (respectively, otherwise) (cf. Lemma \ref{lem-cor}).
The odd class number of $\mathcal{O}_\Delta$ $(\Delta>0)$ deserves special emphasis in both proofs of Theorems \ref{CGPY-thm} and \ref{Miz-thm}.
For a fundamental discriminant $d>0$, it is well-known that $h(d)$ is odd if and only if  
\begin{enumerate}[\hspace{0.7cm}(i)] 
\item either $d=8$, 
\item or $d=p$ for some prime $p\equiv1$ $(\text{mod }4)$,
\item or $d\in\{4p,8p\}$ for some prime $p\equiv3$ $(\text{mod }4)$,
\item or $d=p_1p_2$ for some distinct primes $p_1\equiv p_2\equiv 3$ $(\text{mod }4)$
\end{enumerate}
(cf. \cite[Theorem 5.6.13]{Hal13}).
In this paper, we study a similar problem regarding 
the odd class number of $\mathbb{Q}(\sqrt{p_1p_2})$ for distinct primes $p_1\equiv p_2\equiv 3$ $(\text{mod }4)$.

We state the main results of this paper as follows.
\begin{thm}\label{main-thm}
For distinct primes $p_1\equiv p_2\equiv 3$ $(\text{mod }4)$,
$$h(-p_1)h(-p_2)\equiv 
h(p_1p_2)\frac{\Psi(\omega_{p_1p_2})}{n} \quad(\text{mod } 8),$$
where $n=6$ (respectively, $n=2$) if $\min\{p_1,p_2\}>3$ (respectively, otherwise).
\end{thm}

\begin{thm}\label{main-thm-f2}
For distinct primes $p_1\equiv p_2\equiv 3$ $(\text{mod }4)$,
$$h(-p_1)h(-p_2)\theta(-p_1,-p_2,2)\equiv h(4p_1p_2) \frac{\Psi(\omega_{4p_1p_2})}{n} \quad(\text{mod } 8).$$
Here $n$ is defined as in Theorem \ref{main-thm} and $\theta(-p_1,-p_2,2)$ is defined by
$$\theta(-p_1,-p_2,2):=(2-\chi_{-p_1}(2))(2-\chi_{-p_2}(2))-(1-\chi_{-p_1}(2))(1-\chi_{-p_2}(2))$$ 
with the Kronecker symbols $\displaystyle\chi_{-p_i}:=\big(\frac{-p_i}{\cdot}\big)$.
\end{thm}

\begin{rmk}{\rm
With the same notation as above,
we see that $h(-p_1)h(-p_2)$ is odd by Gauss' genus theory. We also note that $n\mid \Psi(\omega_{p_1p_2})$ and $n\mid\Psi(\omega_{4p_1p_2})$, which can be inferred from \cite[Satz 2]{Lan76} (see also Lemma \ref{unit-str} \eqref{HK-1}).
It is well-known that $\displaystyle\chi_{m}(2)=(-1)^{(m^2-1)/8}$ for an odd integer $m$.}  
\end{rmk}

\vspace{10pt}

The outline of this paper is as follows.
In Section \ref{sec-2}, we recall the Kaneko-Mizuno-Zagier formula and the related notions.
In Section \ref{sec-3}, we review a calculation method of Hirzebruch sums via Dedekinds sums, and we show some facts on the  fundamental units of $\mathcal{O}_{p_1p_2}$ and $\mathcal{O}_{4p_1p_2}$.
In Section \ref{sec-4}, we prove Theorems \ref{main-thm} and \ref{main-thm-f2} according to the cases whether $\varepsilon_{p_1p_2}$ has half-integral coefficients and $f=1$ simultaneously or not.
\vspace{10pt}

\section{The Kaneko-Mizuno-Zagier formula}\label{sec-2}
To state the results of the second named author and Kaneko \cite{KM20}, we recall the genus character on the narrow class group for a general discriminant (for detail, we refer to \cite[Section 3]{BS96} and \cite[Appendix A.2]{KM20}).

By $F=[a,b,c]$ we denote the {\it primitive quadratic form} $F(x,y)=ax^2+bxy+cy^2$
with $a,b,c \in \mathbb{Z}$ such that $\gcd(a,b,c)=1$.
We call $\Delta:=b^2-4ac$ the {\it discriminant} of $F$.
As mentioned in Section \ref{sec-1}, we assume that $\Delta$ is not a square.
The action $F\mapsto AF$ of $A={\footnotesize \begin{pmatrix}\alpha& \beta \\
\gamma &\delta \end{pmatrix}}\in\mathrm{SL}_2(\mathbb{Z})$ is defined by
$(AF)(x,y):=F(\alpha x +\gamma y, \beta x +\delta y)$.
We denote by $[F]$ the {\it equivalence class} of $F$ with respect to the action $F\mapsto AF$ of
$A\in\mathrm{SL}_2(\mathbb{Z})$.
Let $\mathfrak{F}_\Delta$ be the set of all equivalence classes on ``not negative-definite''\footnote{It
means ``positive-definite'' (respectively, ``indefinite'') for $\Delta<0$ (respectively, $\Delta>0$).
}
primitive forms of discriminant $\Delta$.
By $\mathcal{C}_\Delta^+$ (respectively, $\mathcal{C}_\Delta$) we denote the {\it narrow} 
(respectively, {\it wide}) {\it class group} of a quadratic order with discriminant $\Delta$.
We denote by $[\mathfrak{a}]^+$ (respectively, $[\mathfrak{a}]$) the {\it narrow} (respectively, {\it wide}) 
{\it ideal class} of an ideal $\mathfrak{a}$ in $\mathcal{O}_\Delta$.

\begin{prop}[{\cite[Theorem 6.4.2]{Hal13}}]\label{form to narrow}
For a primitive form $F=[a,b,c]$ of discriminant $\Delta$,
let $I(F):=(|a|,\frac{b+\sqrt{\Delta}}{2})\sqrt{\Delta}^{(1-\mathrm{sign}(a))/2}$ be an ideal
of $\mathcal{O}_\Delta$.
The map $\Phi_\Delta: \mathfrak{F}_\Delta \to \mathcal{C}^+_\Delta$ defined by
$$\Phi_\Delta([F]):=[I(F)]^+$$
is bijective.
We define the composition $[F_1]\ast [F_2] \in \mathfrak{F}_\Delta$ by
$$[F_1]\ast [F_2]:=\Phi_\Delta^{-1}(\Phi_\Delta([F_1])\Phi_\Delta([F_2])).$$
With this composition,
$\mathfrak{F}_\Delta$ is an abelian group and $\Phi_\Delta$ is an isomorphism.
\end{prop}

Using Proposition \ref{form to narrow}, we define the genus character as follows.
\begin{defn}[{\cite[Appendix A.2]{KM20}}]\label{def-genus-char}
Let $\Delta=d_1d_2f^2$ be a discriminant, where $d_1$ and $d_2$ are distinct fundamental discriminants and $f$ is a positive integer.
We define the genus character $\chi_{d_1,d_2}^{(\Delta)}:\mathfrak{F}_\Delta \to \{\pm1\}$ as follows.
We write $d_1=q^\ast_{1}q^\ast_{2}\cdots q^\ast_{m}$ with prime fundamental
discriminants\footnote{For
any odd prime $q_i>0$,
$q^\ast_i:=(-1)^{\frac{q_i-1}{2}}q_i$. If $d_1$ is even, $q_1^\ast=2^\ast\in\{-4,\pm8\}$ is defined by
$d_1=q^\ast_{1}q^\ast_{2}\cdots q^\ast_{m}$.
}
$q^\ast_{i}$'s.
For $[F]\in \mathfrak{F}_\Delta$, we define
$$\chi_{d_1,d_2}^{(\Delta)}([F]):=
\prod_{i=1}^{m}
\chi^{(q^\ast_{i})}([a,b,c]),$$
where $[a,b,c]$ is a representative of $[F]$ and
$$\chi^{(q^\ast_{i})}([a,b,c]):=\left\{
\begin{array}{ll}
\chi_{q^\ast_{i}}(a) &\text{ if }\, \gcd(a,q^\ast_i)=1,\\
\chi_{q^\ast_{i}}(c) &\text{ if }\, \gcd(c,q^\ast_i)=1,
\end{array}\right.$$
with $\chi_{q^\ast_{i}}$ the quadratic character of $\mathcal{O}_{q^\ast_{i}}$.
This map is a well-defined homomorphism
and gives the genus character $\chi_{d_1,d_2}^{(\Delta)}:\mathcal{C}^+_\Delta \to \{\pm 1\}$
through $\Phi_\Delta$ in Proposition \ref{form to narrow}.
\end{defn}

We review quadratic irrationals and their continued fractions (for detail, we refer to \cite[Chapter 1 and Chapter 5.5]{Hal13}).
Let $\mathbb{X}_\Delta:=\{\xi=\frac{b+\sqrt{\Delta}}{2a};\, a\neq 0,b,c,\in\mathbb{Z},\,
\gcd(a,b,c)=1,\,b^2-4ac=\Delta\}$
be the set of all {\it quadratic irrationals} of discriminant $\Delta$.
The action $\xi\mapsto A\xi$ of $A={\footnotesize \begin{pmatrix}\alpha& \beta \\
\gamma &\delta \end{pmatrix}}\in\mathrm{GL}_2(\mathbb{Z})$ is defined by
$A\xi:=(\alpha \xi+\beta)/(\gamma \xi + \delta)$.
Two quadratic irrationals $\xi_1$, $\xi_2\in\mathbb{X}_\Delta$ are {\it properly equivalent} 
($\xi_1\sim_+ \xi_2$) if there exists $M\in\mathrm{SL}_2(\mathbb{Z})$ such that $\xi_1=M\xi_2$.
Two quadratic irrationals $\xi_1$, $\xi_2\in\mathbb{X}_\Delta$ are {\it equivalent} 
($\xi_1\sim \xi_2$) if there exists $M\in\mathrm{GL}_2(\mathbb{Z})$ such that $\xi_1=M\xi_2$.
We denote by $\mathfrak{X}^+_\Delta$
(respectively, $\mathfrak{X}_\Delta$) the set of all proper equivalence (respectively, equivalence)
classes of $\mathbb{X}_\Delta$.

\begin{prop}[{\cite[Theorem 5.5.8]{Hal13}}]\label{irrational to narrow}
For a quadratic irrational $\xi=\frac{b+\sqrt{\Delta}}{2a}$, let
$I(\xi):=(|a|,\frac{b+\sqrt{\Delta}}{2})\sqrt{\Delta}^{(1-\mathrm{sign}(a))/2}$ be an ideal of $\mathcal{O}_\Delta$.
The map $\iota^+_\Delta : \mathfrak{X}^+_\Delta \to \mathcal{C}^+_\Delta$ defined by
$$\iota^+_\Delta ([\xi]_{\sim_+})= [I(\xi)]^+$$
is bijective.
Also, the map $\iota_\Delta:\mathfrak{X}_\Delta \to \mathcal{C}_\Delta$ defined by
$\iota_\Delta([\xi]_{\sim})=[I(\xi)]$ is bijective.
\end{prop}

By $\xi'$ we denote the {\it conjugate} of $\xi\in\mathbb{X}_\Delta$.
For $\Delta>0$, let $\mathbb{X}^0_\Delta:=
\{\xi\in\mathbb{X}_\Delta;-1<\xi'<0,\,1<\xi\}$ be the set of all {\it reduced quadratic irrationals}
of discriminant $\Delta$. 
We denote the {\it continued fraction} of $\xi\in\mathbb{R}$ by
$\xi=[u_0,u_1,u_2,\cdots]$ with $u_0\in\mathbb{Z}$ and $u_i\in\mathbb{N}$ for $i\ge1$.
We define the number $\xi_n:=[u_n,u_{n+1},u_{n+2},\cdots]$ for $n\ge0$.
We denote by $\mathcal{N}:\mathbb{Q}(\sqrt{\Delta})\to \mathbb{Q}$ the field norm and 
by $\varepsilon_{\Delta}$ the fundamental unit of $\mathcal{O}_{\Delta}$.
The following proposition provides some properties of continued fractions.
\begin{prop}[{\cite[Theorems 2.2.2, 2.2.9 and 2.3.3]{Hal13}}]\label{prop-cfrac}
Let $\xi=[u_0,u_1,u_2,\cdots]$ be a continued fraction of an irrational $\xi\in\mathbb{R}\setminus \mathbb{Q}$.
\begin{enumerate}[(i)]
\item The sequence $(u_n)_{n\ge 0}$ is ultimately periodic if and only if $\xi$ is a quadratic irrational, and it is periodic if and only if $\xi$ is a reduced quadratic irrational.
\item\label{exist-red} If $\xi\in\mathbb{X}_{\Delta}$, then $\xi_n\in\mathbb{X}_{\Delta}$ and $\xi_n\sim_+ (-1)^n\xi$.
\item\label{all-red-equiv}
Let $\xi=[\hat{v_1},\cdots,\hat{v_k};\overline{v_0,\cdots,v_{l-1}}]\in\mathbb{X}_{\Delta}$.
The $l$ distinct numbers $\xi_k$, $\xi_{k+1}$, $\cdots$, $\xi_{k+l-1}$ are all reduced quadratic irrationals which are equivalent to $\xi$.
\item\label{minus-proper-euqiv}
In the notation above, the period length $l$ is odd
if and only if $\mathcal{N}(\varepsilon_{\Delta})=-1$.
\item\label{inverse-cfrac}(The inverse period) If $\xi=[\overline{v_0,v_1,\cdots,v_{l-1}}]$, then
$-\xi'^{-1}=[\overline{v_{l-1},v_{l-2},\cdots,v_0}]$.
\end{enumerate}
\end{prop}

Now we introduce the main results of \cite{KM20}.
\begin{thm}[{\cite[Lemma 10 (2) and Theorem 3]{KM20}}] \label{thm-KM}
Let $\Delta=d_1d_2f^2$ be a discriminant
with distinct negative fundamental discriminants $d_1$ and $d_2$, and a positive integer $f$.
Let $\varepsilon_\Delta$ be the fundamental unit of $\mathcal{O}_\Delta$ and
$\omega(d_i)$ the number of the roots of unity in $\mathcal{O}_{d_i}$.
We define
$$\displaystyle
\theta(d_1,d_2,f):=\prod_{\substack{p\mid f\\p:\text{ prime}}}
\frac{(1-\chi_{d_1}(p))(1-\chi_{d_2}(p)) - p^{m_p-1}(p-\chi_{d_1}(p))(p-\chi_{d_2}(p))}{
1-p},$$
where $m_p$ is a positive integer such that $p^{m_p}$ 
is the highest power of $p$ dividing $f$ if $f>1$.
The empty product is understood as being $1$ if $f=1$.
Then we have $\mathcal{N}(\varepsilon_\Delta)=1$ and 
\begin{eqnarray}\label{KM-equality}
24\frac{h(d_1)h(d_2)}{\omega(d_1)\omega(d_2)} \theta(d_1,d_2,f)
= \sum_{\xi \in\mathbb{X}_\Delta^0}\chi_{d_1,d_2}^{(\Delta)}([I(\xi)]^+)\lfloor \xi \rfloor,
\end{eqnarray}
where $\chi_{d_1,d_2}^{(\Delta)}$ is defined in
Definition \ref{def-genus-char} and $[I(\xi)]^+$ is defined in Proposition \ref{irrational to narrow}.
\end{thm}

Under the same assumption as in Theorem \ref{thm-KM},
we will restate the equality \eqref{KM-equality}
in terms of $\mathfrak{X}_\Delta$ and the Hirzebruch sum $\Psi$.
Let $\eta:=\frac{b+\sqrt{\Delta}}{2a}\in \mathbb{X}_\Delta$. 
By Propositions \ref{irrational to narrow}, \ref{form to narrow} and $d_1<0$,
we have
\begin{eqnarray}\label{chi-minus-eq}
\chi_{d_1,d_2}^{(\Delta)}([I(-\eta)]^+)&=&\chi_{d_1,d_2}^{(\Delta)}([[-a,b,-c]]) \nonumber\\
&=&\chi_{d_1}(-1)\chi_{d_1,d_2}^{(\Delta)}([[a,b,c]]) \nonumber\\
&=&-\chi_{d_1,d_2}^{(\Delta)}([I(\eta)]^+).
\end{eqnarray}
By Proposition \ref{prop-cfrac} \eqref{exist-red}, \eqref{all-red-equiv} and \eqref{minus-proper-euqiv},
for any $\eta=[\hat{v_1},\cdots,\hat{v_k};\overline{v_0,\cdots,v_{2m-1}}]\in \mathbb{X}_\Delta$, we have
$\{\xi \in \mathbb{X}^0_\Delta \mid \xi \sim_+ (-1)^k\eta \}=\{\eta_{k},\eta_{k+2},\cdots,\eta_{k+2m-4},\eta_{k+2m-2}\}$
and
$\{\xi \in \mathbb{X}^0_\Delta \mid \xi \sim_+ (-1)^{k+1}\eta \}=\{\eta_{k+1},\eta_{k+3},\cdots,\eta_{k+2m-3},\eta_{k+2m-1}\}$.
Hence we have
\begin{eqnarray*}
\sum_{\substack{{\xi \in\mathbb{X}_\Delta^0}\\{\xi\sim \eta}}}
\chi_{d_1,d_2}^{(\Delta)}([I(\xi)]^+)\lfloor \xi \rfloor
&=&\chi_{d_1,d_2}^{(\Delta)}([I(\eta_k)]^+)\Psi(\eta_{k})\\
&=&\left((-1)^{k}\chi_{d_1,d_2}^{(\Delta)}([I(\eta)]^+)\right)
\left((-1)^{k}\Psi(\eta)\right)
\end{eqnarray*}
since $\lfloor\eta_{k+i}\rfloor=v_i$ for $0\le i \le 2m-1$ and \eqref{chi-minus-eq}.
Therefore, we get
\begin{eqnarray}\label{KM-equality-2}
24\frac{h(d_1)h(d_2)}{\omega(d_1)\omega(d_2)} \theta(d_1,d_2,f)
=\sum_{[\eta]_\sim\in\mathfrak{X}_\Delta}\chi_{d_1,d_2}^{(\Delta)}([I(\eta)]^+)\Psi(\eta),
\end{eqnarray}
which we call the {\it Kaneko-Mizuno-Zagier formula}.

\begin{rmk}\label{Remark 2.17}
{\rm
Zagier \cite[(23)]{Zag75} proved
the equality \eqref{KM-equality-2} in the case where $\gcd(d_1,d_2)=1$ and $f=1$.
We remark that Lu \cite{Lu91} obtained a different kind of formula for $h(d_1)h(d_2)$
in the case where $d_1\mid d_2$ and $f=1$ (cf. \cite[Theorem 2.12]{CG19}).
}
\end{rmk}
\vspace{10pt}

\section{Preliminary Lemmas}\label{sec-3}

In this section, we show some lemmas required in the proofs of Theorems \ref{main-thm} and \ref{main-thm-f2}.
\subsection{Hirzebruch sums via Dedekind sums}\label{sec-3.2}

To introduce one of the calculation methods of Hirzebruch sums,
we review some properties of Dedekind sums (for detail, we refer to \cite{RG72}).

\begin{defn}\label{def-DD-sum}
Let $h$ and $k$ be integers such that $\gcd(h,k)=1$ and $k\ge1$.
Let
$$s(h,k):=\sum_{m=1}^{k}\Big(\Big(\frac{hm}{k}\Big)\Big) \Big(\Big(\frac{m}{k}\Big)\Big),$$
where
$$((x)):=\left\{\begin{array}{ll}
x-\lfloor x \rfloor -\tfrac{1}{2} &\text{if }\, x\not\in\mathbb{Z},\\
0 &\text{if }\, x\in\mathbb{Z}.
\end{array}\right.
$$
We call this number $s(h,k)$ the Dedekind sum.
\end{defn}

\begin{prop}\label{Zag75-prop}
Let $h$ and $k$ be integers such that $\gcd(h,k)=1$ and $k\ge1$.
\begin{enumerate}[(i)]
\item\label{h and -h} $s(-h,k)=-s(h,k)$.
\item\label{h-mod-k}$s(h,k)$ depends only on $h$ $(\text{mod }k)$.
\item\label{DD-reciprocity}(The reciprocity law)
If $h\ge1$, then
$\displaystyle s(h,k)+s(k,h)=\tfrac{h^2+k^2+1}{12hk}-\tfrac{1}{4}$.
\item\label{Chu-prop-1}
For an integer $h'$ such that $hh'\equiv 1$ $(\text{mod }k)$,
$s(h,k)=s(h',k)$.
\item\label{Zag75-prop-ii} $6ks(h,k)$ is an integer.
\item\label{Chu-prop-2}
If $k$ is odd, then
$\Big(\frac{h}{k}\Big)=(-1)^{\frac{1}{2}\left(\frac{k-1}{2}-6ks(h,k)\right)},$ where
$\Big(\frac{\phantom{a}}{\phantom{a}}\Big)$ is the {\it Jacobi symbol}.
\end{enumerate}
\end{prop}
\begin{proof}\,
\eqref{h and -h} and \eqref{h-mod-k} follow from Definition \ref{def-DD-sum}.
We refer \eqref{DD-reciprocity} to Theorem 1 in \cite[Chapter 2]{RG72}.
We refer \eqref{Chu-prop-1}, \eqref{Zag75-prop-ii} and \eqref{Chu-prop-2}
to the equation (33c), Theorem 2 and Theorem 4 in \cite[Chapter 3]{RG72}, respectively.
\end{proof}

We define a subset $\fM$ of $\SL_2(\mathbb{Z})$ by
$$\fM:=\left\{M :
M={\footnotesize \begin{pmatrix}x& y \\ z &w\end{pmatrix}} \in\SL_2(\mathbb{Z})
\,\text{ such that }\, z\neq 0  \right\}.$$
For any $M={\footnotesize \begin{pmatrix}x& y \\ z &w\end{pmatrix}} \in\fM$, we define
$$n_M:=\frac{x+w}{z}-\sign(z)(3+12s(w,|z|)),$$
where $\sign(z):=z/|z|$.
The following three lemmas will be used to prove Lemmas \ref{pf-lem-3} and \ref{pf-lem-4} in Section \ref{sec-4}.
\begin{lem}[{\cite[Lemma 8]{KM20}}]\label{KM-lem-8}
Let $M={\footnotesize \begin{pmatrix}x&y\\z&w\end{pmatrix}}\in\fM$
such that $x\neq 0$, and let
$M'={\footnotesize \begin{pmatrix}k&-1\\1&0\end{pmatrix}} M$
for some integer $k$.
Then it follows that
$$
n_{M'}=\left\{
\begin{array}{ll}
k+9+n_M &\text{ if }\, x<0 \,\text{ and }\, z>0,\\
k-3+n_M &\text{ if }\, x>0 \,\text{ or }\, z<0.
\end{array}
\right.
$$
\end{lem}

\begin{lem}\label{translation-matrix}
Let $M={\footnotesize \begin{pmatrix}x&y\\z&w\end{pmatrix}}\in\fM$, and let
$M'={\footnotesize \begin{pmatrix}1&k\\0&1\end{pmatrix}}
M
{\footnotesize \begin{pmatrix}1&-k\\0&1\end{pmatrix}}$\,
for some integer $k$.
Then one has\, $n_{M'}=n_M$.
\end{lem}
\begin{proof}\,
We have $M'={\footnotesize
\begin{pmatrix}x+kz&-kx+y-k^2z+kw\\z&-kz+w\end{pmatrix}}\in\fM$ and
$n_{M'}=\frac{x+w}{z}-\sign(z)(3+12s(-kz+w,|z|))$.
By Proposition \ref{Zag75-prop} \eqref{h-mod-k} we have $n_{M'}=n_M$.
\end{proof}

\begin{lem}\label{minus-n_M}
Let $M={\footnotesize \begin{pmatrix}x&y\\z&w\end{pmatrix}}\in\fM$
such that $x>0$ and $y\neq0$. Let
$M'={\footnotesize \begin{pmatrix}0&1\\1&0\end{pmatrix}}
M
{\footnotesize \begin{pmatrix}0&1\\1&0\end{pmatrix}}$.
Then we have $n_{M'}=-n_M$.
\end{lem}
\begin{proof}\,
We have $M'={\footnotesize \begin{pmatrix}w&z \\ y&x\end{pmatrix}}\in\fM$.
By Proposition \ref{Zag75-prop} \eqref{Chu-prop-1} we have $s(w,|z|)=s(x,|z|)$.
By Proposition \ref{Zag75-prop} \eqref{DD-reciprocity} we have
$s(x,|z|)=-s(|z|,x)+\frac{x^2+z^2+1}{12x|z|}-\frac{1}{4}$.
By Proposition \ref{Zag75-prop} \eqref{h and -h} we have
$n_{M}=\frac{x+w}{z}+12s(z,x)-\frac{x^2+z^2+1}{xz}$.
Similarly, by Proposition \ref{Zag75-prop} \eqref{DD-reciprocity} and \eqref{h and -h}, we have
$n_{M'}=\frac{x+w}{y}+12s(y,x)-\frac{x^2+y^2+1}{xy}$.
By Proposition \ref{Zag75-prop} \eqref{Chu-prop-1} and \eqref{h and -h},
we see that $s(z,x)=s(-y,x)=-s(y,x)$.
Therefore, we have $n_{M}+n_{M'}=\frac{(y+z)(xw-yz-1)}{xyz}=0$.
\end{proof}
\vspace{10pt}

Now we introduce a calculation method of Hirzebruch sums via Dedekind sums.
$\Delta$ always stands for a positive discriminant.
Let $\omega_\Delta:=(\sigma_\Delta+\sqrt{\Delta})/2$, where $\sigma_\Delta=0$ if $\Delta$ is even and $\sigma_\Delta=1$ if $\Delta$ is odd.
We write the fundamental unit
$\varepsilon_\Delta=q+r\omega_\Delta$ with positive integers $q$ and $r$.
For $\xi\in\mathbb{X}_\Delta$, we define $M_\xi$ by a unique integral matrix
$M_\xi:={\footnotesize \begin{pmatrix}x& y \\ z &w\end{pmatrix}}$ such that
$$\left\{ \begin{array}{rll}
\varepsilon_\Delta \xi&=&x\xi+y,\\
\varepsilon_\Delta&=&z\xi+w.
\end{array}\right.
$$
From the definition, we easily see that
\begin{equation}\label{omega-1}
M_{\omega_\Delta}=
{\footnotesize \begin{pmatrix} q+r\sigma_\Delta & r(\Delta-\sigma_\Delta)/4 \\ r & q\end{pmatrix}}.
\end{equation}
For $\xi=\frac{b+\sqrt{\Delta}}{2a}\in\mathbb{X}_\Delta$, we see that
$\xi=
{\footnotesize \begin{pmatrix} 1 & (b-\sigma_\Delta)/2 \\ 0 & a\end{pmatrix}}
\omega_\Delta
$ and
\begin{equation}\label{omega-2}
M_\xi=
{\footnotesize \begin{pmatrix} 1 & (b-\sigma_\Delta)/2 \\ 0 & a\end{pmatrix}}
M_{\omega_\Delta}
{\footnotesize \begin{pmatrix} 1 & (b-\sigma_\Delta)/2 \\ 0 & a\end{pmatrix}}^{-1}.
\end{equation}
We note that $M_\xi\in\fM$ for any $\xi\in\mathbb{X}_\Delta$.
For $\xi\in\mathbb{X}_\Delta$, we define $n(\xi):=n_{M_\xi}$.
\begin{lem}\label{via-cal}
Let $\Delta$ be a positive discriminant such that $\cN(\varepsilon_\Delta)=1$.
For any $\xi\in\mathbb{X}^0_\Delta$, we have
$$\Psi(\xi)=n(\xi).$$
\end{lem}
\begin{proof}\,
The above equality was proved in the middle of the proof of Theorem 2 in \cite[Section 5]{KM20}.
The assumption of \cite[Theorem 2]{KM20} is that
$\Delta$ is of the form $d_1d_2f^2$,
where $d_1$ and $d_2$ are distinct negative fundamental discriminants,
and $f$ is a positive integer.
However, we can see that the above equality still holds under the weaker assumption, that is,
$\cN(\varepsilon_\Delta)=1$.
\end{proof}

\begin{lem}\label{not-reduced}
Let $\Delta$ be a positive discriminant such that $\cN(\varepsilon_\Delta)=1$.
Let $\omega_\Delta=(\sigma_\Delta+\sqrt{\Delta})/2$.
Then we have
$$\Psi(\omega_\Delta)=n(\omega_\Delta).$$
\end{lem}
\begin{proof}\,
We have
$(\omega_{\Delta})_1=1/(\omega_\Delta-\lfloor \omega_\Delta \rfloor)\in\mathbb{X}_\Delta$, where $(\omega_{\Delta})_1$ denotes the shift by 1 as given in the notation of Proposition \ref{prop-cfrac}.
Since $0<\omega_\Delta-\lfloor \omega_\Delta \rfloor<1$ and
$\omega'_\Delta-\lfloor \omega_\Delta \rfloor<-1$, we see that $(\omega_{\Delta})_1\in\mathbb{X}^0_\Delta$.
By Lemma \ref{via-cal} we have $\Psi((\omega_{\Delta})_1)=n((\omega_{\Delta})_1)$.
Form the definition, $\Psi(\omega_\Delta)=-\Psi((\omega_{\Delta})_1)$.
Let $A:={\footnotesize \begin{pmatrix}0&1 \\ 1&0\end{pmatrix}}$ and
$B:={\footnotesize \begin{pmatrix}1&-\lfloor \omega_\Delta \rfloor \\ 0&1\end{pmatrix}}$.
We see that $(\omega_{\Delta})_1=AB\omega_\Delta$ and
$M_{(\omega_{\Delta})_1}=AB M_{\omega_\Delta} B^{-1} A^{-1}$.
By \eqref{omega-1} we have
$$
B M_{\omega_\Delta} B^{-1}=
{\footnotesize \begin{pmatrix}q+r\sigma_\Delta-r\lfloor \omega_\Delta \rfloor
& -\lfloor \omega_\Delta \rfloor^2 + \sigma_\Delta \lfloor \omega_\Delta \rfloor +(\Delta-\sigma_\Delta)/4
\\r
& q+r\lfloor \omega_\Delta \rfloor \end{pmatrix}}\in\fM
$$
and $q+r\sigma_\Delta-r\lfloor \omega_\Delta \rfloor>\varepsilon'_\Delta>0$.
Since $\omega_\Delta$ and $\omega'_\Delta$ are two distinct roots of
$X^2-\sigma_\Delta X +(\sigma_\Delta-\Delta)/4$,
the $(1,2)$-component of $B M_{\omega_\Delta} B^{-1}$ is not equal to $0$.
By Lemma \ref{translation-matrix} and Lemma \ref{minus-n_M},
we have $n_{M_{\omega_\Delta}}=n_{B M_{\omega_\Delta} B^{-1}}
=-n_{AB M_{\omega_\Delta} B^{-1}A^{-1}}=-n_{M_{(\omega_{\Delta})_1}}$.
Therefore, we have
$\Psi(\omega_\Delta)=-\Psi((\omega_{\Delta})_1)=-n((\omega_{\Delta})_1)=n(\omega_\Delta)$.
\end{proof}
\vspace{10pt}

\subsection{Fundamental units of quadratic orders}
The first two lemmas are standard facts on quadratic orders.
\begin{lem}[{\cite[p. 328]{Hal13}}]\label{lem-cor}
Let $\Delta=f^2d$ be a discriminant with a fundamental discriminant $d$ and a conductor $f$. Then the class numbers of $\cO_d$ and $\cO_{\Delta}$ have the following relation:
$$h(f^2d)=h(d)
\frac{f}{[\cO_d^\times:\cO_{f^2d}^\times]}
\displaystyle \prod_{\substack{q\mid f \\ q:\text{ prime}}}\big(1-\chi_d(q)q^{-1}\big).$$
\end{lem}
\begin{lem}\label{unit-str}  
Let $p_1$ and $p_2$ be distinct primes such that $p_1\equiv p_2\equiv 3$ $(\text{mod }4)$. 
Let $\varepsilon_{p_1p_2}=\frac{t+u\sqrt{p_1p_2}}{2}$ be the fundamental unit of $\cO_{p_1p_2}$ such that $t\equiv u$ $(\text{mod }2)$.
Then we have 
\begin{enumerate}[(i)]
\item\label{HK-1} 
$h(p_1p_2)$ is odd and $\cN(\varepsilon_{p_1p_2})=1$. 
\item\label{HK-2}
If $p_1p_2\equiv1$ $(\text{mod }8)$, then $t\equiv u\equiv 0$ $(\text{mod }2)$.
\item\label{HK-3} $\displaystyle \big(h(4p_1p_2),\,\varepsilon_{4p_1p_2}\big)$ is equal to
$$\left\{\begin{array}{rl}
\big(h(p_1p_2),\,\varepsilon_{p_1p_2}\big) & \,\text{ if }\, p_1p_2\equiv 1\,\,\,(\text{mod }8),\\
\big(3h(p_1p_2),\,\varepsilon_{p_1p_2}\big) & \,\text{ if }\, p_1p_2\equiv 5\,\,\,(\text{mod }8)
\,\text{ and }\, t\equiv u \equiv 0\,\,\,(\text{mod }2),\\
\big(h(p_1p_2),\,\varepsilon_{p_1p_2}^3\big) & \,\text{ if }\, p_1p_2\equiv 5\,\,\,(\text{mod }8)
\,\text{ and }\, t\equiv u \equiv 1\,\,\,(\text{mod }2).
\end{array}\right.$$
\end{enumerate}
\end{lem}
\begin{proof}\,
\eqref{HK-1} 
We refer to \cite[Theorem 5.6.13 and Theorem 5.2.2]{Hal13}.\\
\eqref{HK-2} We refer to \cite[Theorem 5.2.3]{Hal13}.\\
\eqref{HK-3} 
If $t\equiv u\equiv 0$ $(\text{mod }2)$, then $\varepsilon_{4p_1p_2}=\varepsilon_{p_1p_2}$.
Now suppose that $t\equiv u\equiv 1$ $(\text{mod }2)$. 
By \cite[Theorem 5.2.3]{Hal13}
we have $p_1p_2\equiv 5$ $(\text{mod }8)$ and $\varepsilon_{4p_1p_2}=\varepsilon_{p_1p_2}^3$.
The results on $h(4p_1p_2)$ follow from Lemma \ref{lem-cor} and $\chi_d(2)=(-1)^{(d^2-1)/8}$ for odd $d$.
\end{proof}

By modifying the results of \cite{Wil15} and \cite{ZY14}, we show the following properties on the fundamental integral units $\varepsilon_{p_1p_2}$ and $\varepsilon_{4p_1p_2}$.
\begin{lem}\label{ZY-thm} 
Let $p_1$ and $p_2$ be distinct primes such that $p_1\equiv p_2\equiv 3$ 
$(\text{mod }4)$. Let $f=2$ if $p_1p_2\equiv 5$ $(\text{mod }8)$ and $\varepsilon_{p_1p_2}$ has half-integral coefficients, and let $f=1$ or $2$ otherwise. 
Let $\Delta=p_1p_2f^2$ and $\varepsilon_{\Delta}=x+y\sqrt{p_1p_2}$ be the fundamental unit of $\cO_{\Delta}$. Then we have 
\begin{enumerate}[(i)]
\item\label{ZY-1} $x$, $y$ are integers such that $x\equiv 7$ $(\text{mod }8)$ and $y\equiv 0$ $(\text{mod }4)$. 
Moreover, $y\equiv 4$ $(\text{mod }8)$ if $(p_1,p_2)\equiv (3,3)$ $(\text{mod }8)$, and
$y\equiv 0$ $(\text{mod }8)$ if $(p_1,p_2)\equiv (7,7)$ $(\text{mod }8)$.
\item\label{ZY-2} $p_1\varepsilon_{\Delta}=(X+Y\sqrt{p_1p_2})^2$ for some integers $X$ and $Y$.
\end{enumerate}
\end{lem}
\begin{proof}\,
\eqref{ZY-1}  We refer to \cite[Theorem 2]{Wil15}.\\ 
\eqref{ZY-2} By \cite[Lemma 3.2]{ZY14} there exist $X$, $Y\in\mathbb{Q}$ such that 
$p_1(x+y\sqrt{p_1p_2})=(X+Y\sqrt{p_1p_2})^2$. 
We have $p_1x=X^2+p_1p_2Y^2$ and $p_1y=2XY$; so, we get 
$$p_1p_2Y^4-p_1xY^2+\frac{p_1^2y^2}{4}=0.$$
Then $Y^2=(x+\alpha)/(2p_2)$ with $\alpha\in\{\pm1\}$ since $x^2-p_1p_2y^2=1$. 
It follows that $2\mid (x+\alpha)$ and $p_2\mid (x+\alpha)$; thus we have that $Y^2\in\mathbb{Z}$ and $X^2=p_1x-p_1p_2Y^2\in\mathbb{Z}$. 
Therefore, $X$ and $Y$ are integers.
\end{proof}

We use Dirichlet's method \cite{Dir1834} (see also \cite{Wil15}) to prove the following lemma about the fundamental half-integral unit $\varepsilon_{p_1p_2}$.
\begin{lem}\label{William-thm}
Let $p_1$ and $p_2$ be distinct primes such that $p_1\equiv p_2\equiv 3$ 
$(\text{mod }4)$. Assume that $p_1p_2\equiv 5$ $(\text{mod }8)$ and 
$\varepsilon_{p_1p_2}=\frac{t+u\sqrt{p_1p_2}}{2}$ with $t\equiv u\equiv 1$ $(\text{mod }2)$. Then we have $t\equiv 1$ $(\text{mod }4)$.
\end{lem}
\begin{proof}\,
We note that $(t,u)\in\mathbb{N}^2$ is the smallest nontrivial solution of 
$x^2-p_1p_2y^2=\pm 4$, and $t^2-p_1p_2u^2=4$.
Since $(t-2)(t+2)=p_1p_2u^2$ and $(t-2,t+2)=1$, 
there exist odd positive integers $r$ and $s$ satisfying one of the following:
$$(t-2,t+2)=\left\{
\begin{array}{l}
(p_1r^2,p_2s^2),\\
(p_2r^2,p_1s^2),\\
(p_1p_2r^2,s^2),\\
(r^2,p_1p_2s^2).
\end{array}\right.$$
For the third case (resp., the fourth case) we have that $s^2-p_1p_2r^2=4$ (resp., 
$r^2-p_1p_2s^2=-4$), which is a contradiction with the minimality of $(t,u)$.
Therefore, we have that
$t=2+p_ir^2\equiv 1$ $(\text{mod }4)$,
where $i=1$ or $2$.
\end{proof}
\vspace{10pt}

\section{Proofs of Theorems \ref{main-thm} and \ref{main-thm-f2}}\label{sec-4}
Let $p_1$ and $p_2$ be distinct primes congruent to $3$ modulo $4$.
Under the same notation as in Theorem \ref{thm-KM}, let $d_1=-p_1$ and $d_2=-p_2$, and let $f=1$ or $2$. 
By the Kaneko-Mizuno-Zagier formula \eqref{KM-equality-2} we have
\begin{eqnarray}\label{app-KM}
&&h(-p_1)h(-p_2)\theta(-p_1,-p_2,f)\nonumber\\
&=&\left\{
\begin{array}{ll}
\displaystyle
 \sum_{[\eta]_\sim\in\fX_{p_1p_2f^2}} \chi^{(p_1p_2f^2)}_{-p_1,-p_2}([I(\eta)]^+)\frac{\Psi(\eta)}{6}
&\displaystyle \text{if }\min\{p_1,p_2\}>3,\\
\displaystyle
 \sum_{[\eta]_\sim\in\fX_{p_1p_2f^2}} \chi^{(p_1p_2f^2)}_{-p_1,-p_2}([I(\eta)]^+)\frac{\Psi(\eta)}{2}
&\displaystyle \text{if }\min\{p_1,p_2\}=3.
\end{array}\right. 
\end{eqnarray}
Moreover, we remark that 
if\, $\min\{p_1,p_2\}>3$ (resp., $\min\{p_1,p_2\}=3$), then
$\Psi(\eta)/6$ (resp., $\Psi(\eta)/2$) is an integer (cf. \cite[Satz 2]{Lan76} and Lemma \ref{unit-str} \eqref{HK-1}).

To show Theorems \ref{main-thm} and \ref{main-thm-f2}, we prove the following four lemmas which are similar to those in \cite{CGPY15} and \cite{Miz21}.

\begin{lem}\label{pf-lem-1}
Following the above notation, for any equivalence class $[\eta]_\sim$ in $\fX_{p_1p_2f^2}$, there exists a representative
$\frac{b+\sqrt{p_1p_2f^2}}{2a}\in\mathbb{X}^0_{p_1p_2f^2}$ of $[\eta]_\sim$
such that $(a,p_1)=1$ and $a>0$. 
\end{lem}
\begin{proof}\,
By Proposition \ref{prop-cfrac} \eqref{all-red-equiv}
there exists a reduced quadratic irrational
$\xi=\frac{\beta+\sqrt{p_1p_2f^2}}{2\alpha}\in\mathbb{X}^0_{p_1p_2f^2}$ such that $\xi\sim \eta$.
Since $\xi'<\xi$, $\alpha$ is a positive integer.
If $(\alpha,p_1)=1$, we take $(a,b)=(\alpha,\beta)$.
Now, we suppose that $p_1$ divides $\alpha$. Let $\gamma=\frac{\beta^2-p_1p_2f^2}{4\alpha}\in\mathbb{Z}$.
We have that $p_1$ divides $\beta$ and $(\gamma,p_1)=1$ since $\gcd(\alpha,\beta,\gamma)=1$.
With the same notation as in Proposition \ref{prop-cfrac}, let $\xi_1=\frac{\beta_1+\sqrt{p_1p_2f^2}}{2\alpha_1}\in\mathbb{X}^0_{p_1p_2f^2}$.
We see that $\alpha_1>0$, $\xi_1\sim\xi\sim\eta$ and
$$\xi_1
= \frac{1}{\xi-\lfloor \xi \rfloor}
=\frac{2\alpha}{\beta-2\alpha\lfloor \xi \rfloor+\sqrt{p_1p_2f^2}}
=\frac{-\beta+2\alpha\lfloor \xi \rfloor+\sqrt{p_1p_2f^2}}{-2\gamma+2\beta\lfloor \xi \rfloor - 2\alpha \lfloor \xi \rfloor^2}.
$$
Therefore, $\alpha_1=-\gamma+\beta\lfloor \xi \rfloor-\alpha\lfloor \xi \rfloor^2$ is coprime to $p_1$.
We take $(a,b)=(\alpha_1,\beta_1)$.
\end{proof}

\begin{lem}\label{pf-lem-2}
Following the above notation,
let $\eta$ and $\delta$ be two quadratic irrationals of discriminant $p_1p_2f^2$
such that $[I(\eta)]=[I(\delta)]^{-1}$ in $\cC_{p_1p_2f^2}$.
Then we have
$$\chi^{(p_1p_2f^2)}_{-p_1,-p_2}([I(\eta)]^+)\Psi(\eta)
=\chi^{(p_1p_2f^2)}_{-p_1,-p_2}([I(\delta)]^+)\Psi(\delta).$$
\end{lem}
\begin{proof}\,
By Lemma \ref{pf-lem-1} there exists $\xi=\frac{b+\sqrt{p_1p_2f^2}}{2a}\in \mathbb{X}^0_{p_1p_2f^2}$
such that $\xi\sim \eta$, $(a,p_1)=1$ and $a>0$.
Let $\xi^\op:=\lfloor \xi \rfloor -\xi'=\lfloor \xi \rfloor+\frac{-b+\sqrt{p_1p_2f^2}}{2a}$.
Since $I(\xi)=(a,\frac{b+\sqrt{p_1p_2f^2}}{2})$ and $I(\xi^\op)=(a,\frac{-b+\sqrt{p_1p_2f^2}}{2})$, we have
$[I(\xi)]=[I(\xi^\op)]^{-1}$ in $\cC_{p_1p_2f^2}$ 
and $\xi^\op\sim \delta$.
By Lemma \ref{lem-cor} \eqref{HK-1}, \eqref{HK-3} and Proposition \ref{prop-cfrac} \eqref{minus-proper-euqiv},
the period length of $\xi$ is even.
We write $\xi=[\overline{v_0,v_1,\cdots,v_{2m-1}}]$.
Then we have 
$\xi^\op=[\overline{v_0,v_{2m-1},v_{2m-2},\cdots,v_2,v_1}]$ by Proposition \ref{prop-cfrac} \eqref{inverse-cfrac}.
We see that $\chi^{(p_1p_2f^2)}_{-p_1,-p_2}([I(\xi)]^+)=\chi_{-p_1}(a)=\chi^{(p_1p_2f^2)}_{-p_1,-p_2}([I(\xi^\op)]^+)$
and $\Psi(\xi)=\sum_{i=0}^{2m-1}(-1)^i v_i=\Psi(\xi^\op)$.
Therefore, we have
$$\chi^{(p_1p_2f^2)}_{-p_1,-p_2}([I(\eta)]^+)\Psi(\eta)
=\chi_{-p_1}(a) \sum_{i=0}^{2m-1}(-1)^i v_i
=\chi^{(p_1p_2f^2)}_{-p_1,-p_2}([I(\delta)]^+)\Psi(\delta)$$
as desired.
\end{proof}

\begin{lem}\label{pf-lem-3}
Following the above notation, we further assume that $f=2$ if $p_1p_2\equiv5$ $(\text{mod }8)$ and $\varepsilon_{p_1p_2}$ has half-integral coefficients, and we assume that $f=1$ or $2$ otherwise.
For any $\eta\in\mathbb{X}_{p_1p_2f^2}$, we have
$$
\chi^{(p_1p_2f^2)}_{-p_1,-p_2}([I(\eta)]^+)\Psi(\eta)
\equiv
\Psi(\omega_{p_1p_2f^2})
\qquad (\text{mod }8).
$$
\end{lem}
\begin{proof}\, 
With the same notation as in the paragraph before Lemma \ref{via-cal}, we write 
$\sigma_{\Delta}$, $\omega_{\Delta}$ and 
$\varepsilon_{\Delta}=q+r\omega_{\Delta}$, where $\Delta=p_1p_2f^2$ and $q$, $r$ are positive integers.
By Lemma \ref{pf-lem-1} there exists $\xi=\frac{b+\sqrt{p_1p_2f^2}}{2a}\in \mathbb{X}^0_{p_1p_2f^2}$
such that $\xi\sim \eta$, $(a,p_1)=1$ and $a>0$.
Let $c=\frac{b^2-p_1p_2f^2}{4a}\in\mathbb{Z}$,
and let $m$ be an integer such that $b=2m+\sigma_\Delta$.
Let $A:={\footnotesize \begin{pmatrix}1& m \\ 0&1\end{pmatrix}}$ and
$B:={\footnotesize \begin{pmatrix}3& -1 \\ 1&0\end{pmatrix}}$. 
By Lemma \ref{unit-str} we have $\cN(\varepsilon_{\Delta})=1$.
By \eqref{omega-1} we have
\begin{align*}
AM_{\omega_{\Delta}}A^{-1}
&=\begin{pmatrix}1&m\\0&1\end{pmatrix}
\begin{pmatrix}q+r\sigma_{\Delta}&r(\Delta-\sigma_{\Delta})/4\\r&q\end{pmatrix}
\begin{pmatrix}1&-m\\0&1\end{pmatrix}\\
&=\begin{pmatrix}q+rm+r\sigma_{\Delta}& -rac\\r&q-rm\end{pmatrix}
\end{align*}
and
$$BAM_{\omega_{p_1p_2}}A^{-1}
=\begin{pmatrix}3(q+rm+r\sigma_{\Delta})-r &-3rac-(q-rm)\\
q+rm+r\sigma_{\Delta} &-rac \end{pmatrix}.$$
By \eqref{omega-2} we have
\begin{align*}
M_{\xi}
&=\begin{pmatrix}1&m\\0&a\end{pmatrix}
\begin{pmatrix}q+r\sigma_{\Delta}&r(\Delta-\sigma_{\Delta})/4\\r&q\end{pmatrix}
\begin{pmatrix}1&-\frac{m}{a}\\0&\frac{1}{a}\end{pmatrix}
\\
&=\begin{pmatrix}q+rm+r\sigma_{\Delta} &-rc\\
ra&q-rm\end{pmatrix}
\end{align*}
and
$$BM_{\xi}
=\begin{pmatrix}3(q+rm+r\sigma_{\Delta})-ra &-3rc-(q-rm)\\
q+rm+r\sigma_{\Delta} &-rc\end{pmatrix}.
$$
Since $\xi=\frac{b+\sqrt{p_1p_2f^2}}{2a}\in\mathbb{X}^0_{p_1p_2f^2}$,
we have $2m+\sigma_{\Delta}>-\sqrt{\Delta}$ and $q+rm +r\sigma_{\Delta}>\varepsilon'_{\Delta}>0$.
By Lemma \ref{not-reduced}, Lemma \ref{translation-matrix} and Lemma \ref{KM-lem-8}, we have
\begin{equation}\label{last-eq-1}
\Psi(\omega_{\Delta})=n(\omega_{\Delta})=n_{AM_{\omega_{\Delta}}A^{-1}}=n_{BAM_{\omega_{\Delta}}A^{-1}}.
\end{equation}
By Lemma \ref{via-cal} and Lemma \ref{KM-lem-8}, we have
\begin{equation}\label{last-eq-2}
\Psi(\xi)=n(\xi)=n_{BM_{\xi}}.
\end{equation}
Since $(a,p_1)=1$ and $\chi^{(p_1p_2f^2)}_{-p_1,-p_2}([I(\xi)]^+)=\chi_{-p_1}(a)$,
by \eqref{last-eq-1} and \eqref{last-eq-2},
it suffice to show that
\begin{eqnarray*}
&&\frac{-r-rac}{q+rm+r\sigma_{\Delta}}-12s(-rac,q+rm+r\sigma_{\Delta})\\
&\equiv&\chi_{-p_1}(a)\Big(\frac{-ra-rc}{q+rm+r\sigma_{\Delta}}-12s(-rc,q+rm+r\sigma_{\Delta})\Big)
\quad (\text{mod }8).
\end{eqnarray*}
By Lemma \ref{ZY-thm} \eqref{ZY-1} we have 
\begin{equation}\label{3mod4}
q+rm+r\sigma_{\Delta}\equiv 3 \quad (\text{mod }4).
\end{equation}  
By Lemma \ref{ZY-thm} \eqref{ZY-1} if $f=1$, then $8\mid r$, and
if $f=2$, then $4\mid r$, $2\mid b$ and at least one of $\{a,c\}$ is odd; so, we have $8\mid r(a-1)(c-1)$ and $8\mid r(a+1)(c+1)$.
By Proposition \ref{Zag75-prop} \eqref{Zag75-prop-ii} it suffices to show that
\begin{eqnarray}\label{lem4.3-suffice-eq}
&&6(q+rm+r\sigma_{\Delta})s(-rac,q+rm+r\sigma_{\Delta}) \nonumber\\
&\equiv& \chi_{-p_1}(a) 6(q+rm+r\sigma_{\Delta})s(-rc,q+rm+r\sigma_{\Delta})\,\, (\text{mod }4).
\end{eqnarray}
By Proposition \ref{Zag75-prop} \eqref{Chu-prop-2} we have
\begin{equation}\label{suffice-eq-2}
\Big(\frac{-rac}{q+rm+r\sigma_{\Delta}}\Big)=(-1)^{\frac{1}{2}\left(\frac{q+rm+r\sigma_{\Delta}-1}{2}-6(q+rm+r\sigma_{\Delta})s(-rac,q+rm+r\sigma_{\Delta})\right)}
\end{equation}
and
\begin{equation}\label{suffice-eq-3}
\Big(\frac{-rc}{q+rm+r\sigma_{\Delta}}\Big)=(-1)^{\frac{1}{2}\left(\frac{q+rm+r\sigma_{\Delta}-1}{2}-6(q+rm+r\sigma_{\Delta})s(-rc,q+rm+r\sigma_{\Delta})\right)}.
\end{equation}
By dividing or multiplying \eqref{suffice-eq-2} by \eqref{suffice-eq-3}, we have 
\begin{equation}\label{suffice-eq-4}
\Big(\frac{a}{q+rm+r\sigma_{\Delta}}\Big)=(-1)^{\frac{1}{2}Z},
\end{equation}
where 
\begin{align*}
Z=&\frac{1-\chi_{-p_1}(a)}{2}(q+rm+r\sigma_{\Delta}-1)-6(q+rm+r\sigma_{\Delta})s(-rac,q+rm+r\sigma_{\Delta})\\
&+6\chi_{-p_1}(a)(q+rm+r\sigma_{\Delta})s(-rc,q+rm+r\sigma_{\Delta}).
\end{align*}
By Lemma \ref{ZY-thm} \eqref{ZY-2} there exist integers $T$ and $U$ such that $p_1 \varepsilon_{\Delta}=(T+U\omega_{\Delta})^2$; so, we have that 
$p_1(q+r\sigma_{\Delta}/2)=(T+U\sigma_{\Delta}/2)^2+U^2\Delta/4$ and $p_1r/2=(T+U\sigma_{\Delta}/2)U$. Then we have
\begin{eqnarray}\label{a-even-eq}
p_1(q+rm+r\sigma_{\Delta})&=&\Big(T+\frac{U\sigma_{\Delta}}{2}\Big)^2+\frac{U^2\Delta}{4}
+\Big(T+\frac{U\sigma_{\Delta}}{2}\Big)Ub \nonumber\\
&=&\Big(T+\frac{U\sigma_{\Delta}}{2}+\frac{Ub}{2}\Big)^2-U^2ac
\end{eqnarray}
and $T+U(\sigma_{\Delta}+b)/2\in\mathbb{Z}$.
Let $e\in\mathbb{Z}$ and $a_1\in\mathbb{Z}\setminus2\mathbb{Z}$ such that $a=2^ea_1$. By \eqref{a-even-eq} we have
$$\Big(\frac{p_1(q+rm+r\sigma_{\Delta})}{a_1}\Big)=1,$$
where $\Big(\frac{\phantom{a}}{\phantom{a}}\Big)$ is the {\it Kronecker symbol} 
(we refer the relation between the Kronecker symbol and the Jacobi symbol to \cite[Section 3.5]{Hal13}). 
We note that $(q+rm+r\sigma_{\Delta},a)=1$ by $\det M_{\xi}=1$. By the law of quadratic reciprocity and \eqref{3mod4}, we have 
\begin{eqnarray}\label{suffice-eq-new}
\Big(\frac{a}{q+rm+r\sigma_{\Delta}}\Big)
&=&(-1)^{\frac{a_1-1}{2}\frac{q+rm+r\sigma_{\Delta}-1}{2}} \Big(\frac{q+rm+r\sigma_{\Delta}}{a}\Big)\nonumber\allowdisplaybreaks\\
&=&(-1)^{\frac{a_1-1}{2}}\Big(\frac{p_1}{a_1}\Big)\Big(\frac{q+rm+r\sigma_{\Delta}}{2^e}\Big).
\end{eqnarray}
We claim that 
\begin{equation}\label{a-even-claim}
\Big(\frac{q+rm+r\sigma_{\Delta}}{2^e}\Big)
=\Big(\frac{p_1}{2^e}\Big).
\end{equation}
If $e$ is even, it is clear. 
If $e\ge 3$, by \eqref{a-even-eq}, then $p_1(q+rm+r\sigma_{\Delta})\equiv 1$ $(\text{mod }8)$; thus we have
$p_1\equiv q+rm+r\sigma_{\Delta}$ $(\text{mod }8)$. 
Now we assume that $e=1$. If $f=2$, we have that $p_1p_2=m^2-ac$, $m$ is odd 
and $4\mid ac$, which is a contradiction with $\gcd(a,b,c)=1$. 
Therefore, we have that $f=1$ and $p_1p_2\equiv 1$ $(\text{mod }8)$. 
We write $\varepsilon_{p_1p_2}=q+r\omega_{p_1p_2}=x+y\sqrt{p_1p_2}$, i.e., $q=x-y$ and $r=2y$. 
By Lemma \ref{ZY-thm} \eqref{ZY-1} we have that 
$q\equiv 3$ $(\text{mod }8)$ and $r\equiv 0$ $(\text{mod }8)$ if $(p_1,p_2)\equiv (3,3)$ $(\text{mod }8)$,  
and that $q\equiv 7$ $(\text{mod }8)$ and $r\equiv 0$ $(\text{mod }8)$ if $(p_1,p_2)\equiv (7,7)$ 
$(\text{mod }8)$; so, we have
$p_1 \equiv q+rm+r$ $(\text{mod }8)$. Therefore, \eqref{a-even-claim} holds. 
By \eqref{suffice-eq-new} and \eqref{a-even-claim}, 
we have
\begin{equation}\label{suffice-eq-new-a}
\Big(\frac{a}{q+rm+r\sigma_{\Delta}}\Big)
=\chi_{-p_1}(a).
\end{equation}
By \eqref{3mod4}, \eqref{suffice-eq-4} and \eqref{suffice-eq-new-a}, we get 
\eqref{lem4.3-suffice-eq}
as desired.
\end{proof}

\begin{lem}\label{pf-lem-4}
Following the above notation, we further assume that $p_1p_2\equiv 5$ $(\text{mod }8)$ and 
$\varepsilon_{p_1p_2}=\frac{t+u\sqrt{p_1p_2}}{2}$ is the fundamental unit of $\cO_{p_1p_2}$ such that $t\equiv u\equiv 1$ $(\text{mod }2)$. 
Then for any $\eta\in\mathbb{X}_{p_1p_2}$, we have
$$
\chi^{(p_1p_2)}_{-p_1,-p_2}([I(\eta)]^+)\Psi(\eta)
\equiv
\Psi(\omega_{p_1p_2})
\qquad (\text{mod }8).
$$
\end{lem}
\begin{proof}\, 
We write $\varepsilon_{p_1p_2}=\frac{t+u\sqrt{p_1p_2}}{2}=q+r\omega_{p_1p_2}$, where $q$, $r$ are positive integers and $r$ is odd. 
We note that $t=2q+r$ and $u=r$.
By Lemma \ref{pf-lem-1} there exists $\xi=\frac{b+\sqrt{p_1p_2}}{2a}\in \mathbb{X}^0_{p_1p_2}$
such that $\xi\sim \eta$, $(a,p_1)=1$ and $a>0$.
Let $c=\frac{b^2-p_1p_2}{4a}\in\mathbb{Z}$,
and let $m$ be an integer such that $b=2m+1$.
It follows that $a$ and $c$ are odd because $p_1p_2\equiv 5$ $(\text{mod }8)$.
Let $A:={\footnotesize \begin{pmatrix}1& m \\ 0&1\end{pmatrix}}$, 
$B:={\footnotesize \begin{pmatrix}3& -1 \\ 1&0\end{pmatrix}}$ 
and 
$C:={\footnotesize \begin{pmatrix}0& 1 \\ 1&0\end{pmatrix}}$.
By \eqref{omega-1} we have
$$BAM_{\omega_{p_1p_2}}A^{-1}
=\begin{pmatrix}3(q+rm+r)-r &-3rac-(q-rm)\\
q+rm+r &-rac \end{pmatrix}$$
and
$$BCAM_{\omega_{p_1p_2}}A^{-1}C^{-1}
=\begin{pmatrix}3(q-rm)+rac &3r-(q+rm+r)\\
q-rm &r \end{pmatrix}.$$
By \eqref{omega-2} we have
$$BM_{\xi}
=\begin{pmatrix}3(q+rm+r)-ra &-3rc-(q-rm)\\
q+rm+r &-rc\end{pmatrix}
$$
and 
$$BCM_{\xi}C^{-1}
=\begin{pmatrix}3(q-rm)+rc &3ra-(q+rm+r)\\
q-rm &ra\end{pmatrix}.
$$
Since $\xi=\frac{b+\sqrt{p_1p_2}}{2a}\in\mathbb{X}^0_{p_1p_2}$,
we have $-\sqrt{p_1p_2}<2m+1<\sqrt{p_1p_2}$; 
so, we have $q+rm +r>\varepsilon'_{p_1p_2}>0$ and $q-rm>\varepsilon'_{p_1p_2}>0$. 
By Lemma \ref{not-reduced}, Lemma \ref{translation-matrix}, Lemma \ref{minus-n_M} and Lemma \ref{KM-lem-8}, we have
\begin{equation}\label{last-eq-1-new}
\Psi(\omega_{p_1p_2})
=n(\omega_{p_1p_2})
=n_{BAM_{\omega_{p_1p_2}}A^{-1}}
=-n_{BCAM_{\omega_{p_1p_2}}A^{-1}C^{-1}}.
\end{equation}
By Lemma \ref{via-cal}, Lemma \ref{minus-n_M} and Lemma \ref{KM-lem-8}, we have
\begin{equation}\label{last-eq-2-new}
\Psi(\xi)=n(\xi)=n_{BM_{\xi}}=-n_{BCM_{\xi}C^{-1}}.
\end{equation}
We note that only one of $\{q+rm+r,q-rm\}$ is odd 
because $r$ is odd.
We have that
$$
(q+rm+r-1)(q-rm-1)
=\Big(\frac{t}{2}-1\Big)^2-(\frac{ub}{2})^2 
=2-t-u^2ac
$$
since $t^2-p_1p_2u^2=4$ and $p_1p_2=b^2-4ac$.
Therefore, by Lemma \ref{William-thm} we have that 
\begin{eqnarray}\label{ac-mod4}
(q+rm+r-1)(q-rm-1)\equiv 0\,\,\,(\text{mod }4)
&\,\Longleftrightarrow\,&
ac\equiv 1\,\,\,(\text{mod }4)\nonumber\\
&\,\Longleftrightarrow\,&
\chi_{-4}(a)=\chi_{-4}(c). \qquad
\end{eqnarray}

\noindent
{\bf{Case 1.}} Suppose that $q+rm+r$ is odd.
Since $(a,p_1)=1$ and $\chi^{(p_1p_2)}_{-p_1,-p_2}([I(\xi)]^+)=\chi_{-p_1}(a)$, by \eqref{last-eq-1-new} and \eqref{last-eq-2-new}, it suffices to show that
\begin{eqnarray*}
&&\frac{-r-rac}{q+rm+r}-12s(-rac,q+rm+r)\\
&\equiv&\chi_{-p_1}(a)\Big(\frac{-ra-rc}{q+rm+r}-12s(-rc,q+rm+r)\Big)
\quad (\text{mod }8).
\end{eqnarray*}
It suffices to show that
\begin{eqnarray}\label{lem4.4-suffice-eq-1}
&&\frac{r(\chi_{-4}(a)-\chi_{-p_1}(a))(\chi_{-4}(c)-\chi_{-p_1}(a))}{2} \nonumber\\
&\equiv&
-6(q+rm+r)s(-rac,q+rm+r) \nonumber\\
&&+\chi_{-p_1}(a)6(q+rm+r)s(-rc,q+rm+r)
\quad (\text{mod }4)
\end{eqnarray}
by Proposition \ref{Zag75-prop} \eqref{Zag75-prop-ii}, 
$4\mid (a-\chi_{-p_1}(a))(c-\chi_{-p_1}(a))$, 
and $\alpha\equiv \chi_{-4}(\alpha)$ $(\text{mod }4)$ for any odd $\alpha$.
By Proposition \ref{Zag75-prop} \eqref{Chu-prop-2} we have
\begin{equation}\label{suffice-eq-2-new}
\Big(\frac{-rac}{q+rm+r}\Big)=(-1)^{\frac{1}{2}\left(\frac{q+rm+r-1}{2}-6(q+rm+r)s(-rac,q+rm+r)\right)}
\end{equation}
and
\begin{equation}\label{suffice-eq-3-new}
\Big(\frac{-rc}{q+rm+r}\Big)=(-1)^{\frac{1}{2}\left(\frac{q+rm+r-1}{2}-6(q+rm+r)s(-rc,q+rm+r)\right)}.
\end{equation}
By dividing or multiplying \eqref{suffice-eq-2-new} by \eqref{suffice-eq-3-new}, we have 
\begin{equation}\label{suffice-eq-4-new}
\Big(\frac{a}{q+rm+r}\Big)=(-1)^{\frac{1}{2}Z},
\end{equation}
where 
\begin{align*}
Z=&\frac{1-\chi_{-p_1}(a)}{2}(q+rm+r-1)-6(q+rm+r)s(-rac,q+rm+r)\\
&+6\chi_{-p_1}(a)(q+rm+r)s(-rc,q+rm+r).
\end{align*}
By Lemma \ref{unit-str} \eqref{HK-3}, Lemma \ref{ZY-thm} \eqref{ZY-2} and $t\equiv u\equiv 1$ $(\text{mod }2)$, there exist integers 
$X$, $Y$, $T$ and $U$ such that $p_1 \varepsilon_{p_1p_2}
=\big(\frac{(t-u\sqrt{p_1p_2})(X+Y\sqrt{p_1p_2})}{2}\big)^2
=(T+U\omega_{p_1p_2})^2$; so, we have that 
$p_1(q+r/2)=(T+U/2)^2+U^2p_1p_2/4$ and $p_1r/2=(T+U/2)U$. Then we have
$$p_1(q+rm+r)=\Big(T+\frac{U}{2}\Big)^2+\frac{U^2p_1p_2}{4}+\Big(T+\frac{U}{2}\Big)Ub
=\Big(T+\frac{U}{2}+\frac{Ub}{2}\Big)^2-U^2ac$$
and 
$$\Big(\frac{p_1(q+rm+r)}{a}\Big)=1.$$
We note that $(q+rm+r,a)=1$ by $\det M_{\xi}=1$. By the law of quadratic reciprocity, we have 
\begin{eqnarray}\label{suffice-eq-new-new}
\Big(\frac{a}{q+rm+r}\Big)
&=&(-1)^{\frac{a-1}{2}\frac{q+rm+r-1}{2}} \Big(\frac{q+rm+r}{a}\Big) \nonumber\allowdisplaybreaks\\
&=&\left\{
\begin{array}{ll}
\big(\frac{p_1}{a}\big)
&\text{if }q+rm+r-1\equiv 0\,\, (\text{mod }4),\\
(-1)^{\frac{a-1}{2}} \big(\frac{p_1}{a}\big)
&\text{if }q+rm+r-1\equiv 2\,\, (\text{mod }4),
\end{array}\right.
\nonumber\allowdisplaybreaks\\
&=&\left\{
\begin{array}{ll}
\chi_{-4}(a)\chi_{-p_1}(a)
&\text{if }q+rm+r-1\equiv 0\,\, (\text{mod }4),\\
\chi_{-p_1}(a)
&\text{if }q+rm+r-1\equiv 2\,\, (\text{mod }4).
\end{array}\right.
\end{eqnarray}
By \eqref{ac-mod4}, \eqref{suffice-eq-4-new} and
\eqref{suffice-eq-new-new}, we get \eqref{lem4.4-suffice-eq-1}
as desired.\\

\noindent
{\bf{Case 2.}} Suppose that $q-rm$ is odd.
By \eqref{last-eq-1-new} and \eqref{last-eq-2-new}, it suffices to show that
\begin{eqnarray*}
\frac{rac+r}{q-rm}-12s(r,q-rm)
\equiv \chi_{-p_1}(a)\Big(\frac{ra+rc}{q-rm}-12s(ra,q-rm)\Big)
\quad (\text{mod }8).
\end{eqnarray*}
It suffices to show that
\begin{eqnarray}\label{lem4.4-suffice-eq-2}
&&\frac{r(\chi_{-4}(a)-\chi_{-p_1}(a))(\chi_{-4}(c)-\chi_{-p_1}(a))}{2} \nonumber\\
&\equiv&
6(q-rm)s(r,q-rm)-\chi_{-p_1}(a)6(q-rm)s(ra,q-rm)
\quad (\text{mod }4) \qquad\quad
\end{eqnarray}
by Proposition \ref{Zag75-prop} \eqref{Zag75-prop-ii}, 
$4\mid (a-\chi_{-p_1}(a))(c-\chi_{-p_1}(a))$, 
and $\alpha\equiv \chi_{-4}(\alpha)$ $(\text{mod }4)$ for any odd $\alpha$.
By Proposition \ref{Zag75-prop} \eqref{Chu-prop-2} we have
\begin{equation}\label{suffice-eq-2-new-re}
\Big(\frac{r}{q-rm}\Big)=(-1)^{\frac{1}{2}\left(\frac{q-rm-1}{2}-6(q-rm)s(r,q-rm)\right)}
\end{equation}
and
\begin{equation}\label{suffice-eq-3-new-re}
\Big(\frac{ra}{q-rm}\Big)=(-1)^{\frac{1}{2}\left(\frac{q-rm-1}{2}-6(q-rm)s(ra,q-rm)\right)}.
\end{equation}
By dividing or multiplying \eqref{suffice-eq-2-new-re} by \eqref{suffice-eq-3-new-re}, we have 
\begin{eqnarray}\label{suffice-eq-4-new-re}
\Big(\frac{a}{q-rm}\Big)
=(-1)^{\frac{1}{2}\left(\frac{1-\chi_{-p_1}(a)}{2}(q-rm-1)-6(q-rm)s(r,q-rm)+6\chi_{-p_1}(a)(q-rm)s(ra,q-rm)\right)}.\nonumber\\
\end{eqnarray}
By Lemma \ref{unit-str} \eqref{HK-3}, Lemma \ref{ZY-thm} \eqref{ZY-2} and $t\equiv u\equiv 1$ $(\text{mod }2)$, there exist integers $X$, $Y$, $T$ and $U$ such that 
$p_1 \varepsilon_{p_1p_2}
=\big(\frac{(t-u\sqrt{p_1p_2})(X+Y\sqrt{p_1p_2})}{2}\big)^2
=(T+U\omega_{p_1p_2})^2$; 
so, we have that 
$p_1(q+r/2)=(T+U/2)^2+U^2p_1p_2/4$ and $p_1r/2=(T+U/2)U$. Then we have
$$p_1(q-rm)=\Big(T+\frac{U}{2}\Big)^2+\frac{U^2p_1p_2}{4}-\Big(T+\frac{U}{2}\Big)Ub
=\Big(T+\frac{U}{2}-\frac{Ub}{2}\Big)^2-U^2ac$$
and 
$$\Big(\frac{p_1(q-rm)}{a}\Big)=1.$$
We note that $(q-rm,a)=1$ by $\det M_{\xi}=1$. By the law of quadratic reciprocity, we have 
\begin{eqnarray}\label{suffice-eq-new-new-re}
\Big(\frac{a}{q-rm}\Big)
&=&(-1)^{\frac{a-1}{2}\frac{q-rm-1}{2}} \Big(\frac{q-rm}{a}\Big)
\nonumber\allowdisplaybreaks\\
&=&\left\{
\begin{array}{ll}
\big(\frac{p_1}{a}\big)
&\text{if }q-rm-1\equiv 0\,\, (\text{mod }4),\\
(-1)^{\frac{a-1}{2}} \big(\frac{p_1}{a}\big)
&\text{if }q-rm-1\equiv 2\,\, (\text{mod }4),
\end{array}\right.
\nonumber\allowdisplaybreaks\\
&=&\left\{
\begin{array}{ll}
\chi_{-4}(a)\chi_{-p_1}(a)
&\text{if }q-rm-1\equiv 0\,\, (\text{mod }4),\\
\chi_{-p_1}(a)
&\text{if }q-rm-1\equiv 2\,\, (\text{mod }4).
\end{array}\right.
\end{eqnarray}
By \eqref{ac-mod4}, \eqref{suffice-eq-4-new-re} and
\eqref{suffice-eq-new-new-re}, we get \eqref{lem4.4-suffice-eq-2}
as desired.
\end{proof}

\begin{proof}[Proofs of Theorems \ref{main-thm} and \ref{main-thm-f2}]\,
Let $f=1$ or $2$, and let $n=6$ if $\min\{p_1,p_2\}>3$ and $n=2$ otherwise. By \eqref{app-KM} we have
\begin{eqnarray*}
&&h(-p_1)h(-p_2)\theta(-p_1,-p_2,f)-\frac{h(p_1p_2f^2)\Psi(\omega_{p_1p_2f^2})}{n}\\
&=&\frac{1}{n} \sum_{[\eta]_\sim\in\fX_{p_1p_2f^2}}
\left( \chi^{(p_1p_2f^2)}_{-p_1,-p_2}([I(\eta)]^+)\Psi(\eta) - \Psi(\omega_{p_1p_2f^2}) \right).
\end{eqnarray*}
For $\xi\in\mathbb{X}^0_{p_1p_2f^2}$, let $\xi^\op:=\lfloor \xi \rfloor -\xi'$.
In the proof of Lemma \ref{pf-lem-2}, we
see that $[I(\xi)]=[I(\xi^\op)]^{-1}$ in $\cC_{p_1p_2f^2}$.
By Lemma \ref{unit-str} $h(p_1p_2f^2)$ is odd; so, for any $\xi\in\mathbb{X}^0_{p_1p_2f^2}$ such that 
$\xi$ is not equivalent to $\omega_{p_1p_2f^2}$, we have $[I(\xi)]\neq [I(\xi^\op)]$.
Therefore, there exist representatives $\omega_{p_1p_2f^2}$, $\xi_{(1)}$, $\xi_{(2)}$, $\cdots$, $\xi_{(m)}$
such that
$$\fX_{p_1p_2f^2}=\{[\omega_{p_1p_2f^2}]_\sim,[\xi_{(1)}]_\sim,[\xi_{(1)}^\op]_\sim,\cdots,
[\xi_{(m)}]_\sim,[\xi_{(m)}^\op]_\sim\},$$
where $m=\frac{h(p_1p_2f^2)-1}{2}$.
By Lemma \ref{pf-lem-2} we have
\begin{eqnarray*}
&&h(-p_1)h(-p_2)\theta(p_1,p_2,f)-\frac{h(p_1p_2f^2)\Psi(\omega_{p_1p_2f^2})}{n}\\
&=&\frac{2}{n} \sum_{i=1}^{m}
\left( \chi^{(p_1p_2f^2)}_{-p_1,-p_2}([I(\xi_{(i)})]^+)\Psi(\xi_{(i)}) - \Psi(\omega_{p_1p_2f^2}) \right).
\end{eqnarray*}
Therefore, by Lemma \ref{pf-lem-3} or Lemma \ref{pf-lem-4} and $2\parallel n$, we have
$$h(-p_1)h(-p_2)\theta(-p_1,-p_2,f)-h(p_1p_2f^2)\frac{\Psi(\omega_{p_1p_2f^2})}{n}\equiv 0 \quad (\text{mod }8)$$
as desired.
\end{proof}

\section*{Appendix}
We denote odd primes by $p_1$ and $p_2$. For the small discriminant of $\mathbb{Q}(\sqrt{p_1p_2})$ having the class number 1, 3 or 5, we list the factorizations of 
$$H(f):=H(-p_1,-p_2,f):=h(p_1p_2f^2)\frac{\Psi(\omega_{p_1p_2f^2})}{n}
-h(-p_1)h(-p_2)\theta(-p_1,-p_2,f).$$ 
Here $f$ is 1 or 2,
$h(\Delta)$ is the class number of the quadratic order with discriminant $\Delta$, 
$\Psi$ is the Hirzebruch sum, 
$\omega_{p_1p_2}=(1+\sqrt{p_1p_2})/2$, 
$\omega_{4p_1p_2}=\sqrt{p_1p_2}$,  
$n=6$ (resp., $n=2$) if $\min\{p_1,p_2\}>3$ (resp., if $\min\{p_1,p_2\}=3$), $\theta(1):=\theta(-p_1,-p_2,1):=1$ and $\theta(2):=\theta(-p_1,-p_2,2)
:=\big(2-\big(\frac{-p_1}{2}\big)\big)
\big(2-\big(\frac{-p_2}{2}\big)\big)
-\big(1-\big(\frac{-p_1}{2}\big)\big)
\big(1-\big(\frac{-p_1}{2}\big)\big)$ with the Kronecker symbol $\big(\frac{\phantom{1}}{\phantom{1}}\big)$. 
We denote the fundamental unit of $\mathbb{Q}(\sqrt{p_1p_2})$ by $x_{p_1p_2}+y_{p_1p_2}\omega_{p_1p_2}$.

\vspace{20pt}

\FloatBarrier
\begin{table}[!h]
{Table A1. Fatorizations of $H(-p_1,-p_2,f)$ \,for\, $h(p_1p_2)=1$ and\, $p_1p_2\le 161$}
{\tabcolsep6pt\begin{tabular}{lccccl}
\hline
$\Delta=p_1p_2f^2$ 
&$\Psi(\omega_{\Delta})$
&$(x_{p_1p_2},y_{p_1p_2})$
&$h(\Delta),h(-p_1),h(-p_2)$ 
&$\theta(f)$
&$H(f)$\\ 
\hline
$21=3\cdot 7$ &2 &(2,1) &1,1,1  &1 &0\\
$21\cdot 2^2$ &6 &\, &1,1,1  &3 &0\\
\hline
$33=3\cdot 11$ &2 &(19,8) &1,1,1  &1 &0\\
$33\cdot 2^2$ &10 &\, &1,1,1  &5 &0\\
\hline
$57=3\cdot 19$ &2 &(131,40) &1,1,1  &1 &0\\
$57\cdot 2^2$ &10 &\, &1,1,1  &5 &0\\
\hline
$69=3\cdot 23$ &6 &(11,3) &1,1,3  &1 &0\\
$69\cdot 2^2$ &18 &\, &1,1,3  &3 &0\\
\hline
$77=7\cdot 11$ &6 &(4,1) &1,1,1  &1 &0\\
$77\cdot 2^2$ &18 &\, &1,1,1  &3 &0\\
\hline
$93=3\cdot 31$ &6 &(13,3) &1,1,3  &1 &0\\
$93\cdot 2^2$ &18 &\, &1,1,3  &3 &0\\
\hline
$129=3\cdot 43$ &2 &(15371,2968) &1,1,1  &1 &0\\
$129\cdot 2^2$ &10 &\, &1,1,1  &5 &0\\
\hline
$133=7\cdot 19$ &6 &(79,15) &1,1,1  &1 &0\\
$133\cdot 2^2$ &18 &\, &1,1,1  &3 &0\\
\hline
$141=3\cdot 47$ &10 &(87,16) &1,1,5  &1 &0\\
$141\cdot 2^2$ &26 &\, &3,1,5  &3 
&
$2^3\cdot 3$\\
\hline
$161=7\cdot 23$ &18 &(10847,1856) &1,1,3  &1 &0\\
$161\cdot 2^2$ &18 &\, &1,1,3  &1 &0\\
\hline
\end{tabular}}{}
\end{table}

\FloatBarrier
\begin{table}[!h]
{Table A2. Fatorizations of $H(-p_1,-p_2,f)$ \,for\, $h(p_1p_2)=3$ and\, $p_1p_2\le 1509$}
{\tabcolsep6pt\begin{tabular}{lccccl}
\hline
$\Delta=p_1p_2f^2$ 
&$\Psi(\omega_{\Delta})$
&$(x_{p_1p_2},y_{p_1p_2})$
&$h(\Delta),h(-p_1),h(-p_2)$
&$\theta(f)$
&$H(f)$\\ 
\hline
$321=3\cdot 107$ &18 &(203,24) &3,1,3  &1 
&
$2^3\cdot 3$\\
$321\cdot 2^2$ &42 &\, &3,1,3  &5 
&
$2^4\cdot 3$\\
\hline
$469=7\cdot 67$ &18 &(31,3) &3,1,1  &1 
&
$2^3$\\
$469\cdot 2^2$    &54 &\,        &3,1,1  &3 
&
$2^3\cdot 3$\\
\hline
$473=11\cdot 43$ &18 &(83,8) &3,1,1  &1 
&
$2^3$\\
$473\cdot 2^2$     &42 &\,        &3,1,1  &5 
&
$2^4$\\
\hline
$993=3\cdot 331$ &18 &(2563,168) &3,1,3  &1 
&
$2^3\cdot 3$\\
$993\cdot 2^2$     &42 &\,                &3,1,3  &5 
&
$2^4\cdot 3$\\
\hline
$1101=3\cdot 367$ &22 &(177,11) &3,1,9  &1 
&
$2^3\cdot 3$\\
$1101\cdot 2^2$     &66 &\,                &3,1,9  &3 
&
$2^3\cdot 3^2$\\
\hline
$1257=3\cdot 419$ &38 &(98539,5720) &3,1,9  &1 
&
$2^4\cdot 3$\\
$1257\cdot 2^2$     &94 &\,                &3,1,9  &5 
&
$2^5\cdot 3$\\
\hline
$1509=3\cdot 503$ &46 &(246,13) &3,1,21  &1 
&
$2^4\cdot 3$\\
$1509\cdot 2^2$     &138 &\,                &3,1,21  &3 
&
$2^4\cdot 3^2$\\
\hline
\end{tabular}}{}
\end{table}

\FloatBarrier
\begin{table}[!h]
{Table A3. Fatorizations of $H(-p_1,-p_2,f)$ \,for\, $h(p_1p_2)=5$ and\, $p_1p_2\le 3997$}
{\tabcolsep6pt\begin{tabular}{lccccl}
\hline
$\Delta=p_1p_2f^2$ 
&$\Psi(\omega_{\Delta})$
&$(x_{p_1p_2},y_{p_1p_2})$
&$h(\Delta),h(-p_1),h(-p_2)$
&$\theta(f)$
&$H(f)$\\ 
\hline
$817=19\cdot 43$ &30 &(331,24) &5,1,1  &1 
&
$2^3\cdot 3$\\
$817\cdot 2^2$ &42 &\,                 &5,1,1  &5 
&
$2^4$\\
\hline
$1393=7\cdot 199$ &30 &(3487,192) &5,1,9  &1 
&
$2^4$\\
$1393\cdot 2^2$ &78 &\,                 &5,1,9  &1 
&
$2^3\cdot 7$\\
\hline
$1641=3\cdot 547$ &30 &(4267,216) &5,1,3  &1 
&
$2^3\cdot 3^2$\\
$1641\cdot 2^2$ &54 &\,                 &5,1,3  &5 
&
$2^3\cdot 3\cdot 5$\\
\hline
$1897=7\cdot 271$ &42 &(42895,2016) &5,1,11 &1 
&
$2^3\cdot 3$\\
$1897\cdot 2^2$ &90 &\,                 &5,1,11  &1 
&
$2^6$\\
\hline
$3997=7\cdot 571$ &54 &(280,9) &5,1,5 &1 
&
$2^3\cdot 5$\\
$3997\cdot 2^2$ &162 &\,             &5,1,5  &3 
&
$2^3\cdot 3\cdot 5$\\
\hline
%
\end{tabular}}{}
\end{table}

\bigskip

\noindent Institute of Mathematical Sciences,

\noindent Ewha Womans University,

\noindent Seoul, Korea

\noindent E-mail: jigu.kim.math@gmail.com\bigskip

\noindent Graduate School of Technology Industrial and Social Sciences,

\noindent Tokushima University,

\noindent Tokushima, Japan

\noindent E-mail: mizuno.yoshinori@tokushima-u.ac.jp

\end{document}